\documentclass[10pt,a4paper,oneside]{article}
\usepackage[english]{babel}
\usepackage{a4wide,amsmath,amsthm,amssymb,url,graphicx,xspace,algorithm,algorithmic,pgf}

\usepackage{enumerate}
\usepackage{multirow}
\usepackage[hidelinks]{hyperref}
\usepackage{tabularx}
\usepackage{bm}

\def\F{\mathbb{F}}

\def\k{k_0}

\DeclareMathOperator{\AG}{AG}

\theoremstyle{definition}
\newtheorem{theorem}{Theorem}[section]
\newtheorem{lemma}[theorem]{Lemma}
\newtheorem{definition}[theorem]{Definition}
\newtheorem{remark}[theorem]{Remark}
\newtheorem{corollary}[theorem]{Corollary}

\newtheorem{example}[theorem]{Example}

\newcommand{\comments}[1]{}

\author{Maarten De Boeck \and Geertrui Van de Voorde}
\title{A note on large Kakeya sets}
\date{}
\begin{document}
\maketitle

\begin{abstract}
	A {\em Kakeya set} $\mathcal{K}$ in an affine plane of order $q$ is the point set covered by a set $\mathcal{L}$ of $q+1$ pairwise non-parallel lines. Large Kakeya sets were studied by Dover and Mellinger; in \cite{dm} they showed that Kakeya sets with size at least $q^{2}-3q+9$ contain a large {\em knot} (a point of $\mathcal{K}$ lying on many lines of $\mathcal{L}$). 
	\par In this paper, we improve on this result by showing that Kakeya set of size at least $\approx q^{2}-q\sqrt{q}+\frac{3}{2}q$ contain a large knot. Furthermore, we obtain a sharp result for planes of square order containing a Baer subplane.
\end{abstract}


\section{Introduction}

\par An \emph{affine plane} of order $q$ is a set of $q^{2}$ points and $q^{2}+q$ lines such that every two points lie on exactly one line and any two lines meet in at most one point. It is not too hard to see that the line set of an affine plane of order $q$ can be partitioned into $q+1$ sets of $q$ pairwise disjoint lines, called \emph{parallel classes}. The parallel classes correspond to the so-called points at infinity of the affine plane. We denote the Desarguesian affine plane of order $q$ by $\AG(2,q)$; it is defined over the finite field $\F_{q}$. Affine planes are classical objects in finite geometry; from a design theory point of view they are $2-(q^{2},q,1)$ designs.
\par A \emph{Kakeya set} (also called a minimal Besicovitch set) of an affine plane is the set of points covered by a set of lines  having exactly one line in each parallel class. The study of such sets is the finite field equivalent of the classical euclidean Kakeya problem (see \cite[Section 1.3]{T} for a short survey). The finite field Kakeya problem was first posed in \cite{W}. For $n$-dimensional Kakeya sets (which we do not introduce here) an important result (solving the main question from \cite{W}) was proved in \cite{D}, which was later improved in \cite{ZKS,SS1}.
\par The main research question on Kakeya sets in affine planes is to classify the smallest and largest examples. Most results were obtained for the Desarguesian affine plane. It is easy to see that a Kakeya set $\mathcal{K}$ in an affine plane of order $q$ has size at least $\frac{q(q+1)}{2}$. This bound is met if and only if $q$ is even and the $q+1$ lines defining the Kakeya set form a dual hyperoval together with the line at infinity. The existence of (dual) hyperovals in arbitrary affine and projective planes of even order is a major open question in finite geometry. Hyperovals are known to exist in the Desarguesian affine plane (the classical example consists of a conic and its nucleus, but many examples are known, see \cite{cher} for an overview), but also in some other affine planes, see \cite{dtz,opp}.
\par For the Desarguesian plane we also have the following results. Theorem \ref{even}, dealing with the even order case, is from \cite{bdms} and continues research from \cite{bb}. Theorem \ref{odd}, from \cite{bm}, deals with the odd order case.

\begin{theorem}[{\cite[Theorem 3.6]{bdms}}]\label{even}
	If $\mathcal{K}$ is a Kakeya set in $\AG(2,q)$, $q>8$ even, then only the following possibilities can occur.
	\begin{itemize}
		\item $|\mathcal{K}|=\frac{q(q+1)}{2}$ and $\mathcal{K}$ arises from a dual hyperoval.
		\item $|\mathcal{K}|=\frac{q(q+2)}{2}$ and $\mathcal{K}$ arises from a dual oval and a line not extending it to a dual hyperoval.
		\item $|\mathcal{K}|=\frac{q(q+2)}{2}+\frac{q}{4}$ and $\mathcal{K}$ arises from a KM-arc of type $4$ as described in \cite[Example 3.1]{bdms}.
		\item $|\mathcal{K}|\geq\frac{q(q+2)}{2}+\frac{q}{4}+1$.
	\end{itemize}
\end{theorem}

\begin{theorem}[\cite{bm}]\label{odd}
	If $\mathcal{K}$ is a Kakeya set in $\AG(2,q)$, $q$ odd, then $|\mathcal{K}|\geq \frac{q(q+1)}{2}+\frac{q-1}{2}$. Equality holds if and only if $\mathcal{K}$ arises from a dual oval and one additional line.
\end{theorem}

\par Examples of small Kakeya sets that do not arise from (hyper)ovals or KM-arcs were constructed in \cite{dm2}.
\par At the other end of the spectrum of possible sizes of Kakeya sets, it is easy to see that the maximal size is $q^{2}$ and that this bound is met if and only if the underlying set of lines is the set of $q+1$ lines through a fixed point. More generally, Dover and Mellinger proved the following result. Note that this result holds for any affine plane, not only for Desarguesian affine planes.

\begin{theorem}[{\cite[Theorem 3.3]{dm}}]\label{old}
	Let $\mathcal{K}$ be a Kakeya set in an affine plane of order $q>12$. If $|\mathcal{K}|\geq q^{2}-3q+9$, then $|\mathcal{K}|=q^{2}-kq+k^{2}$ for some $k\in\{0,1,2,3\}$, and $\mathcal{K}$ has a $(q+1-k)$-knot.
\end{theorem}

In this theorem a $j$-\emph{knot} is a point of the Kakeya set that is on $j$ lines of the underlying line set, but not on $j+1$ of them. The main result of this paper is an improvement of Theorem \ref{old}.

\section{The main result}

Our main result, Theorem \ref{maintheoremalt}, gives a clear description of the largest Kakeya sets and shows that there are gaps in the spectrum of admissible sizes of large Kakeya sets. Note that Theorem \ref{old} described an interval of size $\approx 3q$. Theorem \ref{maintheoremalt} describes an interval of size $\approx q\sqrt{q}$ as we will see in Corollary \ref{maintheorem}. We first define and investigate some functions that we will use in the proof of this main theorem.

\begin{definition}
	We define
	\begin{align*}
		f_{q}(k)=\frac{qk}{q+1-k}\;,\qquad\qquad g_{q}(k)=k(q-k)
	\end{align*}
	as functions on $\left[0,q\right]$ and we denote $\k=\sqrt{q+\frac{1}{4}}-\frac{1}{2}$. 
\end{definition}

\begin{lemma}\label{fqgq}
	We have the following results for the functions $f_{q}$ and $g_{q}$:
	\begin{enumerate}[(i)]
		\item $f_{q}$ is increasing on the whole domain,
		\item $g_{q}(k)=g_{q}(q-k)$,
		\item $f_q\left(q-\k\right)=g_q\left(q-\k\right)$ and $f_{q}(k)\leq g_{q}(k)$ if $k\leq q-\k$,
		 \item if $k\in [s,q-s]$ for some $s\in \left[0,\frac{q}{2}\right]$ then $g_q(s)\leq g_q(k)$,
		\item $g_{q}\left(q-\k\right)\leq\min\left\{g_{q}\left(q-\left\lceil\k\right\rceil\right),f_{q}\left(q-\left\lfloor\k\right\rfloor\right)\right\}$ if $q\geq4$.
	\end{enumerate}
\end{lemma}
\begin{proof}
	 Statements (i) and (iii) follow from straightforward calculations. Statements (ii) and (iv) follow easily using the fact that the graph of $g_q$ is a downward opening parabola. We now prove the final statement. It is obvious that $g_{q}\left(q-\k\right)=f_{q}\left(q-\k\right)\leq f_{q}\left(q-\left\lfloor\k\right\rfloor\right)$ since $f_{q}$ is increasing on $[0,q]$. 
	If $q\geq 4$, then $\left\lceil\k\right\rceil\in\left[\k,q-\k\right]$, which implies by (v) that $g_q(\k)\leq g_q\left(\left\lceil\k\right\rceil\right)$, and hence, by (ii), that $g_{q}\left(q-\k\right)\leq g_{q}\left(q-\left\lceil\k\right\rceil\right)$.
\end{proof}

\begin{theorem}\label{maintheoremalt}
	If $\mathcal{K}$ is a Kakeya set in an affine plane of order $q$ and 
	\begin{align*}
		|\mathcal{K}|&> q^2-\min\left\{g_{q}\left(q-\left\lceil\k\right\rceil\right),f_{q}\left(q-\left\lfloor\k\right\rfloor\right)\right\}=q^{2}-\min\left\{\left\lceil \k\right\rceil\left(q-\left\lceil \k\right\rceil\right),\frac{q\left(q-\left\lfloor\k\right\rfloor\right)}{\left\lfloor\k\right\rfloor+1}\right\}\;,
	\end{align*}
	then $\mathcal{K}$ contains a $(q+1-k)$-knot for some $k$ with $0\leq k< \left\lceil
	\k\right\rceil$ and
	\begin{align*}
		|\mathcal{K}|\in\left[q^{2}-kq+\frac{k(k+1)}{2},q^{2}-kq+k^{2}\right]\;.
	\end{align*}
\end{theorem}
\begin{proof}
	Let $\mathcal{L}$ be the affine line set that determines $\mathcal{K}$. By $x_{i}$ we denote the number of points that is contained in precisely $i$ lines of $\mathcal{L}$, $i=0,\dots,q+1$, i.e. the number of $i$-knots. Note that $x_{0}$ is the number of points not in $\mathcal{K}$.
	\par Using the standard countings of the points, the tuples $(P,\ell)$, with $P$ a point on $\ell\in\mathcal{L}$ and the tuples $(P,\ell,m)$, with $\{P\}=\ell\cap m$ and $\ell,m\in\mathcal{L}$, we find
	\begin{align}\label{standardcountingsalt}
		\sum_{i=0}^{q+1}x_{i}=q^{2}\qquad\sum_{i=0}^{q+1}ix_{i}=q^{2}+q\qquad\sum_{i=0}^{q+1}i(i-1)x_{i}=q^{2}+q\;.
	\end{align}
	Now, let $k$ be the integer such that $x_{q+1-k}>0$ but $x_{i}=0$ for all $i>q+1-k$. In other words, $\mathcal{K}$ admits at least one $(q+1-k)$-knot, but no larger knots. It follows that
	\begin{align*}
		0\leq \sum_{i=1}^{q+1}(q+1-k-i)(i-1)x_{i}\;.
	\end{align*}
	This implies that
	\begin{align*}
		0\leq x_{0}(q+1-k)+\sum_{i=0}^{q+1}(q+1-k-i)(i-1)x_{i}.
	\end{align*}
	Using \eqref{standardcountingsalt}, we find that $\sum_{i=0}^{q+1}(q+1-k-i)(i-1)x_{i}$ equals $-qk$, and hence
	\begin{align}\label{ondergrensalt}
		x_{0}\geq \frac{qk}{q+1-k}=f_{q}(k)\;.
	\end{align}
	Let $P$ be a $(q+1-k)$-knot and let $\mathcal{M}$ be the set of $k$ lines through $P$ that are not contained in $\mathcal{L}$. Any point that is not on a line of $\mathcal{L}$ (i.e. not in $\mathcal{K}$) must be on one of the lines of $\mathcal{M}$. Let $\mathcal{L}'$ be the subset of $\mathcal{L}$ containing the $k$ lines not through $P$.
	\par Let $m$ be a line of $\mathcal{M}$. Exactly one of the $k$ lines of $\mathcal{L}'$ is parallel to $m$. Hence, the other $k-1$ each intersect $m$ in (possibly coinciding) points on $m\setminus \{P\}$. This implies that at least $(q-1)-(k-1)$ points on $m$ are not covered by a line of $\mathcal{L}'$ and we find that
	\begin{align}\label{ondergrens2alt}
		x_{0}\geq k(q-k)=g_{q}(k)\;.
	\end{align}
	On the other hand, the lines of $\mathcal{L}'$ cover at least $\frac{k(k-1)}{2}$ points that are on a line of $\mathcal{M}$, different from $P$: the first line covers $k-1$ points, the second line at least $k-2$, and so on. Hence,
	\begin{align}\label{bovengrensalt}
		x_{0}\leq k(q-1)-\frac{k(k-1)}{2}=kq-\frac{k(k+1)}{2}.
	\end{align}
	\par By the assumption in the statement of the theorem we know that $x_0<g_{q}\left(q-\left\lceil\k\right\rceil\right)$ and $x_0<f_{q}\left(q-\left\lfloor\k\right\rfloor\right)$, since $x_{0}$ is smaller than the minimum of both. Recall that $k$ is the smallest integer such that there is a $(q+1-k)-$knot and that we have $x_0\geq f_q(k)$ by \eqref{ondergrensalt} and $x_0\geq g_q(k)$ by \eqref{ondergrens2alt}. Now suppose to the contrary that $k\geq\left\lceil\k\right\rceil$.
	\par We distinguish between two cases. If $k\leq q-\left\lceil\k\right\rceil$, then $k\in \left[\left\lceil\k\right\rceil,q-\left\lceil\k\right\rceil\right]$ and we know from Lemma \ref{fqgq} (ii) and (iv) that $g_{q}\left(q-\left\lceil\k\right\rceil\right)=g_{q}\left(\left\lceil\k\right\rceil\right)\leq g_{q}(k)$. Combining the previous inequalities yields
	\begin{align*}
		g_{q}(k)\leq x_0<g_{q}\left(q-\left\lceil\k\right\rceil\right)=g_{q}\left(\left\lceil\k\right\rceil\right)\leq g_{q}(k)\;,
	\end{align*}
	a contradiction. If $k> q-\left\lceil\k\right\rceil$, then $k\geq q-\left\lfloor\k\right\rfloor$, and we have $f_ {q}\left(q-\left\lfloor\k\right\rfloor\right)\leq f_{q}(k)$ since $f_{q}$ is increasing by Lemma \ref{fqgq} (i). Combining the previous inequalities now yields
	\begin{align*}
		f_{q}(k)\leq x_0<f_{q}\left(q-\left\lfloor\k\right\rfloor\right)\leq f_{q}(k)\;,
	\end{align*}
	which is also a contradiction.
	\par We conclude that if $x_{0}<\min\left\{g_{q}\left(q-\left\lceil\k\right\rceil\right),f_{q}\left(q-\left\lfloor\k\right\rfloor\right)\right\}$, the largest knot is a $(q+1-k)$-knot for some integer $k$ with $k<\left\lceil\k\right\rceil$. It now follows from \eqref{ondergrens2alt} and \eqref{bovengrensalt} that
	\begin{align*}
		|\mathcal{K}|&=q^{2}-x_{0}\leq q^{2}-kq+k^{2}\text{ and}\\
		|\mathcal{K}|&=q^{2}-x_{0}\geq q^{2}-kq+\frac{k(k+1)}{2}\;.\qedhere
	\end{align*}
\end{proof}

\begin{corollary}\label{maintheorem}
	Let $\mathcal{K}$ be a Kakeya set in an affine plane of order $q$, $q\geq 4$. If
	\begin{align*}
		|\mathcal{K}|>q^{2}-f_{q}\left(q-\k\right)= q^{2}-\left((q+1)\sqrt{q+\frac{1}{4}}-\frac{3q+1}{2}\right)\;,
	\end{align*}
	then $\mathcal{K}$ contains a $(q+1-k)$-knot for some $k$ with $0\leq k< \sqrt{q+\frac{1}{4}}-\frac{1}{2}$ and
	\begin{align*}
		|\mathcal{K}|\in\left[q^{2}-kq+\frac{k(k+1)}{2},q^{2}-kq+k^{2}\right]\;.
	\end{align*}
\end{corollary}
\begin{proof}
	We recall from Lemma \ref{fqgq} (iii) and (v) that $f_{q}\left(q-\k\right)=g_{q}\left(q-\k\right)$ and that $f_{q}\left(q-\k\right)=g_{q}\left(q-\k\right)\leq\min\left\{g_{q}\left(q-\left\lceil\k\right\rceil\right),f_{q}\left(q-\left\lfloor\k\right\rfloor\right)\right\}$ if $q\geq4$. The corollary follows from Theorem \ref{maintheoremalt} using these observations.
	
\end{proof}

\begin{remark}
	Note that the bounds of Theorem \ref{maintheoremalt} are in general an improvement on those of Corollary \ref{maintheorem}. For example, for $q=17$, Corollary \ref{maintheorem} states a classification for the Kakeya sets with size at least $17^{2}-48$, but  using Theorem \ref{maintheoremalt}, we can see that we actually have a classification for the Kakeya sets with size at least $17^{2}-51$.
\end{remark}

\begin{remark}
	We see that the spectrum of admissible sizes for the largest Kakeya sets is the union of several intervals. At the upper end of the spectrum there are two intervals that contain only one integer. For $k=0$, we obtain that $|\mathcal{K}|=q^2$ and we already know that every Kakeya set of size $q^{2}$ arises from $q+1$ lines through a point. For $k=1$, we have that $|\mathcal{K}|=q^{2}-q+1$ and that the Kakeya sets of this size arise from a set of $q$ lines through a fixed line, and one line parallel to (but different from) the unique line through the $q$-knot not contained in the underlying line set.
	\par In case of a $(q-1)$-knot Theorem \ref{maintheoremalt} states that the size of the Kakeya set is either $q^{2}-2q+3$ or $q^{2}-2q+4$, but it is easy to see that only $q^{2}-2q+4$ is admissible; this result is also part of Theorem \ref{old}.
	\par The next interval, corresponding to the case of a $(q-2)$-knot, is $[q^{2}-3q+6,q^{2}-3q+9]$. We elaborate briefly on the case that there is a $(q-2)$-knot and $|\mathcal{K}|=q^{2}-3q+6$. Using the notation from the proof of Theorem \ref{maintheoremalt}, we can see that the 6 lines in $\mathcal{L}'\cup\mathcal{M}$ and their intersection points form an affine plane of order 2; together with the line at infinity and the three points on it on the lines of $\mathcal{L}'$, it forms a Fano plane. So, this case can only occur if the affine plane has a subplane of order 2. The Desarguesian affine plane $\AG(2,q)$ admits a subplane of order 2 if and only if $q$ is even. Note that it has been conjectured by Neumann that any non-Desarguesian projective plane admits a Fano subplane (see \cite{neu} for her work related to this conjecture and \cite{tait} and its references for some recent results).
\end{remark}

\begin{remark}
	The description of the spectrum as a union of intervals suggests that there are gaps in the spectrum of admissible sizes. For example it is clear that the $q-2$ integers between $q^{2}-q+1$ and $q^{2}$ cannot be Kakeya set sizes, so for $q\geq3$, there is a gap in the spectrum. For small values of $q$ (up to 9), the full spectrum has been determined in \cite[Table 1]{dm}; for $q=9$ this table only covers the Desarguesian affine plane.
	\par Using the notation from the proof of Theorem \ref{maintheoremalt}, we note that $q^{2}-h_{q}(k)>q^{2}-g_{q}(k+1)$ if $k>\sqrt{2q+\frac{1}{4}}-\frac{3}{2}$. So, if $k>\sqrt{2q+\frac{1}{4}}-\frac{3}{2}$ the intervals $\left[q^{2}-h_{q}(k+1),q^{2}-g_{q}(k+1)\right]$ and $\left[q^{2}-h_{q}(k),q^{2}-g_{q}(k)\right]$, described in the theorem, do not overlap. Further, we note that $q^{2}-h_{q}(k)>q^{2}-g_{q}(k+1)+1$ if $k>\sqrt{2q-\frac{7}{4}}-\frac{3}{2}$. So, if $k>\sqrt{2q-\frac{7}{4}}-\frac{3}{2}$ there is a nontrivial gap between the intervals $\left[q^{2}-h_{q}(k+1),q^{2}-g_{q}(k+1)\right]$ and $\left[q^{2}-h_{q}(k),q^{2}-g_{q}(k)\right]$.
	\par Finally, note that $\sqrt{2q-\frac{7}{4}}-\frac{3}{2}>\left(\sqrt{q+\frac{1}{4}}-\frac{1}{2}\right)-1$ 	for $q\geq3$. So, there is a gap between all intervals described in the statement of Theorem \ref{maintheoremalt}. Note that we are looking at the interval $\left[q^{2}-h_{q}(k+1),q^{2}-g_{q}(k+1)\right]$, so we require $k+1<\sqrt{q+\frac{1}{4}}-\frac{1}{2}$.
\end{remark}

\begin{corollary}\label{square}
	Let $\mathcal{K}$ be a Kakeya set in an affine plane of order $q$, with $q\geq9$ square. If $$|\mathcal{K}|> q^{2}-q\sqrt{q}+q,$$ then $\mathcal{K}$ contains a $(q+1-k)$-knot  with $0\leq k\leq \sqrt{q}$ and $|\mathcal{K}|\in\left[q^{2}-kq+\frac{k(k+1)}{2},q^{2}-kq+k^{2}\right]$.\end{corollary}
\begin{proof}
	Since $q$ is a square, $\sqrt{q}$ is an integer. We apply Theorem \ref{maintheoremalt} where we have  $\left\lfloor\k\right\rfloor=\sqrt{q}-1$ and $\left\lceil \k\right\rceil=\sqrt{q}$, so $\min\left\{\left\lceil \k\right\rceil(q-\left\lceil \k\right\rceil),\frac{q\left(q-\left\lfloor\k\right\rfloor\right)}{\left\lfloor\k\right\rfloor+1}\right\}=\min\{\sqrt{q}(q-\sqrt{q}),\sqrt{q}(q-\sqrt{q}+1)=q\sqrt{q}-q$.\end{proof}

We now give an example of a Kakeya set of size $q^{2}-q\sqrt{q}+q$ for which the largest knots are $(\sqrt{q}+1)$-knots; hypothetically, if Theorem \ref{maintheoremalt} could be extended to Kakeya sets of this size, it would force the existence of a $(q+1-\sqrt{q})$-knot. Furthermore, note that the lower bound in Corollary \ref{maintheorem}, $q^{2}-\left((q+1)\sqrt{q+\frac{1}{4}}-\frac{3q+1}{2}\right)\approx q^{2}-q\sqrt{q}+\frac{3}{2}q-\frac{9}{8}\sqrt{q}+\frac{1}{2}$, is very close to $q^{2}-q\sqrt{q}+q$, so an improvement of Corollary \ref{maintheorem} with purely combinatorial means seems unlikely to succeed.

\begin{example}\label{baer}
	Let $\mathcal{P}$ be a projective plane of order $q$, $q$ a square, that admits a subplane $\mathcal{B}$ of order $\sqrt{q}$ (e.g. the Desarguesian plane). Let $\ell$ be any line in $\mathcal{P}$ that meets $\mathcal{B}$ in precisely one point $P$, and let $\mathcal{A}$ be the affine plane we get by removing $\ell$ and its points from $\mathcal{P}$.
	\par Let $\mathcal{L}'$ be the set of $q$ lines not through $P$ that are extended lines of $\mathcal{B}$, and let $m$ be a line through $P$ that is an extended line of $\mathcal{B}$; there are $\sqrt{q}+1$ choices for $m$. Then $\mathcal{L}=\mathcal{L}'\cup\{m\}$ is a set of $q+1$ lines that defines a Kakeya set $\mathcal{K}$. Since any point of $\mathcal{A}\setminus\mathcal{B}$ is on precisely one line that is an extended subline of $\mathcal{B}$, we find that only the $\sqrt{q}(q-\sqrt{q})$ points of $\mathcal{A}$, lying on an extended line through $P$, different from $m$, are not on a line of $\mathcal{L}$. Hence $|\mathcal{K}|=q^{2}-q\sqrt{q}+q$. This Kakeya set has $q$ points that are $\sqrt{q}$-knots, $\sqrt{q}$ points that are $(\sqrt{q}+1)$-knots, and no larger knots.
\end{example}

Given Example \ref{baer}, we see that the bound in Corollary \ref{square} is indeed sharp. Hence for planes of square order containing a Baer subplane, our classification is the best possible.

\paragraph*{Acknowledgements:} This research was partially carried out when the first author was visiting the School of Mathematics and Statistics at the University of Canterbury. He wants to thank the School, and in particular the second author, for their hospitality. 

We thank the anonymous reviewers for their detailed suggestions.

\noindent Maarten De Boeck\\
Ghent University\\
Department of Mathematics: Algebra and Geometry\\
Krijgslaan 281--S25\\
9000 Gent\\
Flanders, Belgium\\
maarten.deboeck@ugent.be
\\

\noindent Geertrui Van de Voorde\\
University of Canterbury\\
School of Mathematics and Statistics\\
Private Bag 4800\\
8140 Christchurch\\
New Zealand\\
geertrui.vandevoorde@canterbury.ac.nz
\end{document}